\documentclass[preprint,12pt,3p]{elsarticle}

\usepackage{amssymb}

\usepackage{amsthm}
\usepackage{graphicx}
\usepackage{amsmath}
\usepackage{amsfonts}
\usepackage{hyperref}
\usepackage{xcolor}
\usepackage{dsfont}
\usepackage{scalerel}

\usepackage{slashbox}

\usepackage{pgf,tikz,calc}
\usetikzlibrary{math}

\usetikzlibrary{arrows}
\usetikzlibrary{shapes}
\usepackage{multirow}
\usepackage{latexsym}
\usepackage{setspace}
\PassOptionsToPackage{boxed,section}{algorithm}
\usepackage{algorithm}
\usepackage{algorithmic}
\usepackage{enumerate}
\usepackage{amsmath}			
\usepackage{amssymb}			
\usepackage{url}
\usepackage{setspace}
\usetikzlibrary{arrows}
\usetikzlibrary{shapes}
\usetikzlibrary{decorations.markings}
\usepackage{geometry}
\usepackage{bbm}

\newtheorem{theorem}{\em Theorem}
\newtheorem{proposition}[theorem]{\em Proposition}
\newtheorem{conjecture}{\em Conjecture}
\newtheorem{definition}{\em Definition}
\newtheorem{lemma}[theorem]{\em Lemma}

\newtheorem{corollary}[theorem]{\em Corollary}

\newtheorem{observation}[theorem]{\em Observation}

\journal{Journal of Computational and Applied Mathematics}

\begin{document}

\begin{frontmatter}

\title{Adjacent Vertex Distinguishing Total Coloring of Corona Product of Graphs\footnote{Conflict of Interest:  Author A declares that she has no conflict of interest. Author B has received research grants from UNAM - grant  PAPIIT-UNAM-IN117219}}

\author[label1]{Hanna Furma\'nczyk\corref{cor1}}
\address[label1]{Institute of Informatics, Faculty of Mathematics, Physics and Informatics,\\University of Gda\'nsk, Wita Stwosza 57, 80-309 Gda\'nsk, Poland}
 
\ead{hanna.furmanczyk@ug.edu.pl}
\cortext[cor1]{Corresponding author}
\author[label5]{Rita Zuazua}
\address[label5]{Department of Mathematics, Faculty of Sciences, 
National Autonomous University of Mexico, 
Ciudad Universitaria, Coyoacan, 04510 Mexico, DF, Mexico}

\ead{ritazuazua@ciencias.unam.mx}

\begin{abstract}
An adjacent vertex distinguishing total $k$-coloring $f$ of a graph $G$ is a proper total $k$-coloring of $G$ such that no pair of adjacent vertices has the same color sets, where the color set at a vertex $v$, $C^G_f(v)$, is $\{f(v)\} \cup \{f(vu)|u \in V (G), vu \in E(G)\}$. In 2005 Zhang et al. posted the conjecture (AVDTCC) that every simple graph $G$ has adjacent vertex distinguishing total $(\Delta(G)+3)$-coloring. In this paper we confirm the conjecture for many coronas, in particular for generalized, simple and $l$-coronas of graphs, not relating the results to particular graph classes.
\end{abstract}

\end{frontmatter}
\newpage
\section{Introduction}
The processes occurring in the world around us can very often be modeled by the language of graph theory. The graph coloring problems, vertex, edge as well as total version, are ones of the best known problems of graph theory.
Proper total coloring was considered for the first time by Rosenfeld in 1970 \cite{rosen}. In the first decade of this century a new concept appeared in the topic of graph colorings. Many researchers considered colorings (proper, total or from lists) such that vertices (all or adjacent) are distinguished either by sets or multisets or sums. In this paper we investigate the problem of proper total distinguishing adjacent vertices by sets.

Let $G = (V, E)$ be a simple graph with maximum degree $\Delta(G)$. Let $[k]$ denote the set $\{1,\ldots,k\}$ for any positive integer $k$. Suppose that $f : V \cup E \rightarrow
[k]$ is a proper total coloring of $G$, i.e. no two adjacent edges, no two adjacent vertices, and no edge and its endvertices are assigned the same color. The smallest number $k$ admitting such a proper total $k$-coloring is named \emph{total chromatic number} and is denoted by $\chi''(G)$. Clearly, $\chi''(G)\geq \Delta(G)+1$. Vizing \cite{Vizing}, and independly Behzad et al. \cite{behzad}, posted the following conjecture.

\begin{conjecture}[TCC, \cite{behzad,Vizing}]
For any graph $G$, $\chi''(G) \leq \Delta(G)+2.$\label{tcc:conj}
\end{conjecture}

For a given proper total $k$-coloring of $G$ and for a vertex $v\in V(G)$, let $C_f(v)$ denote the \emph{color set} of $v$ with respect to $f$, i.e. the set $\{f(v)\} \cup \{f(vu)|u \in V (G), vu \in E(G)\}$. Sometimes we will consider the color set restricted to some subgraph of $G$. Let $H$ be a subgraph of $G$ and let $v\in V(H)$. Then, $C^H_f(v)$ denotes the set $\{f(v)\} \cup \{f(vu)|u \in V (H), vu \in E(H)\}$. If the total coloring is clear, we can use the notation of $C(v)$ and $C^H(v)$, respectively.

In this paper, we are interested in the smallest number $k$ of colors such that there is a proper total $k$-coloring of $G$ with the adjacent vertices being distinguished by their color sets. Such a model was introduced by Zhang et al. \cite{zhang} in 2005. More formally, in \emph{adjacent vertex distinguishing total
$k$-coloring} $f$ (avd total $k$-coloring, for short), we have $C_f(u) \neq C_f (v)$ for every pair of vertices $u,v$ such that $uv \in E(G)$. The smallest $k$ admitting such coloring  is called the \emph{adjacent vertex distinguishing total chromatic number} (avd total chromatic number, for short) and is denoted by $\chi_{a}''(G)$. Of course, $\chi_{a}''(G)\geq \chi''(G)$. It turns out that there are a lot of examples of graphs for which this inequality is strict, e.g. $2l+1=\chi''(K_{2l+1})<\chi_a''(K_{2l+1})=2l+3$ \cite{short}.

As a direct consequence of the definition, we have the following relation between avd total chromatic number and chromatic number, $\chi(G)$, and chromatic index of a graph $G$, $\chi'(G)$.
\begin{proposition}
For any graph $G$, $\chi''_a(G)\leq \chi(G) + \chi'(G)$.\label{prop:obs}\hfill$\Box$
\end{proposition}
Taking into account for example Vizing and Brook's theorems we get 
\begin{proposition}
Let $G\neq K_n$ and $G\neq C_{2k+1}$. Then $\chi''_a(G)\leq 2\Delta(G)+1.$\hfill$\Box$
\end{proposition}

Huang et al. \cite{huang} proved that the bound from the last proposition can be improved to $2\Delta(G)$.
Whereas, for planar graphs Proposition \ref{prop:obs} implies
\begin{proposition}
    For any planar graph $G$ we have $\chi''_a(G)\leq \Delta(G)+5.$\hfill$\Box$

\end{proposition}

Zhang et al. in \cite{zhang} determined $\chi_{a}''(G)$ for many basic families of graphs, including cycles, complete graphs, fans, wheels or trees. Additionally, the following bound on $\chi_{a}''(G)$ in terms of the maximum degree of a graph $\Delta(G)$ was conjectured.

\begin{conjecture}[AVDTCC, \cite{zhang}]
For any simple graph $G$, $\chi_{a}''(G)\leq \Delta(G)+3$.\label{conj:zhang}
\end{conjecture}

The conjecture has been attracting the attention of many graph theorists since 2005.  It has been proved for some families of graphs, including planar \cite{planar8, planar10, planar9, planfrom10}, outerplanar \cite{outer}, subcubic \cite{delta3, short, delta3wang}, bipartite \cite{delta3}, and 4-regular graphs \cite{4reg}.
The last result is proved by giving a relevant algorithm for avd total 7-coloring of 4-regular graphs. Coker and Johannson \cite{coker} used probabilistic methods to prove that there exists a constant $c$ such that $\chi''_a(G)\leq \Delta(G)+c$. 

Zhang et al. \cite{zhang} proved also 
\begin{lemma}[\cite{zhang}]
If $G$ has two vertices of maximum degree which are adjacent, then $\chi_{a}''(G) \geq \Delta(G) + 2$.
\end{lemma}

\begin{lemma}[\cite{zhang}]
If $G$ has $m$ components $G_i$, $i\in[m]$, and
$|V(G_i)| \geq 2$, $i\in[m]$, then $\chi_{a}''(G) = \max\{\chi_{a}''(G_1), \chi_{a}''(G_2), \ldots, \chi_{a}''(G_m)\}$.
\end{lemma}

That is why we assume that all graphs considered in this paper are connected.

In this paper we put our attention to graph products. They are interesting and useful
in many situations. The complexity of many problems that deal with very large and complicated graphs is reduced greatly if one is able to fully characterize 
the properties of less complicated prime factors.
In the literature we have some results concerning adjacent vertex distinguishing total coloring for join graphs of paths with cycles and fans \cite{join_paths_cycles, join_path_fan}, and some Cartesian products of simple graphs \cite{cartpath, cart, cart_myc, relations}. We consider the objective problem for corona product of graphs - generalized, simple and $l$-coronas. They are often close to the boundary between easy and hard coloring problems \cite{harder}. 
\begin{definition} \emph{For a given simple graph $G$ with $V(G)=\{v_1,\ldots,v_{n_G}\}$, and graphs $H_1,\ldots,$ $H_{n_G}$, the} generalized corona, \emph{denoted by $G\Tilde{\circ} \Lambda _{i=1}^{n_G} H_i$ or by $G \Tilde{\circ} (H_1,\ldots,H_{n_G})$, is the graph obtained by taking one copy of graphs $G$, $H_1,...,H_{n_G}$ and joining the $v_i$ vertex of $G$ to every vertex of $H_i$ (cf. Fig.~\ref{fig:ex_gen}).}
\end{definition}

\begin{figure}
    \centering
    \includegraphics[scale=0.5]{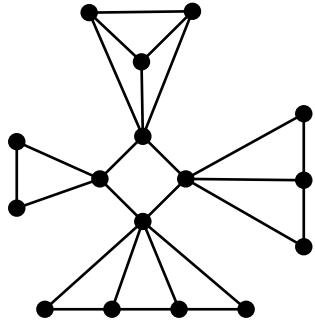}
    \caption{An example of a generalized corona - graph $C_4\Tilde{\circ} (C_3,P_3,P_4,P_2)$.}
    \label{fig:ex_gen}
\end{figure}

In the cases when all graphs $H_i$ are isomorphic, i.e. $H_1 \simeq H_2 \simeq \cdots \simeq H_{n_G}\simeq H$, the generalized corona is reduced to \emph{simple} corona $G\circ H$. Graph $G$ is called the \emph{center graph}, while graph $H$ is named the \emph{outer graph}. Such type of graph product was introduced by Frucht and Harary \cite{frucht}.

Let $v\in V(G)$, by $F_v$ we denote the set of edges of $G\circ H$ linking $v$ with the relevant copy of $H$, and we name it by \emph{fan} in $v$. 

\begin{definition}
\emph{For any integer $l \geq 2$, the graph $G \circ ^l H$ is defined as 
$G \circ ^l H = (G \circ ^{l-1} H ) \circ H$, where $G \circ ^1 H =G \circ H$. Graph $G \circ ^l H$ is also named as} $l$-corona product \emph{of $G$ and $H$. }
\end{definition}

In this paper we confirm Conjecture \ref{conj:zhang} for many coronas, generalized, simple, or $l$-coronas, of graphs that fulfill the conjecture, not relating the results to particular graph classes. We conclude the paper with some open questions.

\section{Main results}\label{main}
We start with some basic observations.
\begin{observation}
Let $G$ be a simple graph. If there is avd total $k$-coloring of $G$ then there is also its avd total $(k+1)$-coloring.
\end{observation}

\begin{observation}
In any avd total $k$-coloring $f$ of a graph $G$, if the degree of a vertex $u$ is different of the degree of $v$, $u, v\in V(G)$, then $C_f(u) \neq C_f (v)$.
\end{observation}

Note that every bipartite graph $G$ has $(\Delta(G)+2)$-adjacent-vertex-\-disting\-uishing total coloring such that $\Delta(G)$ colors are used to color edges of $G$, while the remaining two colors are used to color vertices. Moreover, we may extend this coloring into adjacent-vertex-distinguishing total coloring that uses more than $\Delta(G)+2$ colors by introducing new colors only for vertices.

\begin{theorem}[\cite{delta3}]
If $G$ is a bipartite graph, then $\chi''_a(G)\leq \Delta(G)+2$.\label{avd-bip}
\end{theorem}

We start our main results from the theorem concerning generalized coronas $G\Tilde{\circ} \Lambda _{i=1}^n H_i$ where the maximum degree of each graph $H_i$ does not exceed the maximum degree of $G$. It is worth emphasizing that in the assumption of the theorem we use only avd total coloring of graph $G$ while graphs $H_i$ are colored totally in any proper way.

\begin{theorem}
Let $G$ and $H_1,H_2,\ldots H_{n_G}$ be connected simple graphs, each one on at least two vertices, such that $\chi''_a(G)\leq \Delta(G)+t$, $t\geq 2$, and $\chi''(H_i)\leq \Delta(H_i)+t_i$, $1\leq t_i\leq t$, for each $i\in[n_G]$. In addition, let $\Delta (G)\geq \Delta (H_i)$ for all $1\leq i\leq n_G$. Then,
$$\chi_a''(G\Tilde{\circ} \Lambda _{i=1}^n H_i) \leq \Delta(G \Tilde{\circ} \Lambda _{i=1}^n H_i)+t.$$\label{thm:gen_cor}
\end{theorem}
\begin{proof}
Without loss of generality we can assume that $\Delta(H_1)\geq \Delta(H_2) \geq \cdots\geq \Delta(H_{n_G})$. 
Let $|V(G)|=n_G$ and $|V(H_i)|=n_{H_i}$, $i\in[n_G]$. From the assumption, we have $n_G\geq 2$ and $n_{H_i}\geq 2$ for every $i\in[n_G]$. It is clear that
$\delta(G)+\min_i n_{H_i} \leq \Delta(G\Tilde{\circ} \Lambda _{i=1}^n H_i)\leq \Delta(G)+\max_i n_{H_i}$ and the maximum degree of the generalized corona can be realized only by vertices of $G$. 

In order to obtain adjacent vertex distinguishing total $(\Delta(G\Tilde{\circ} \Lambda _{i=1}^n H_i)+t)$-coloring $f$ we start from any avd total $(\Delta (G)+t)$-coloring of $G$. We will extend the coloring for all graphs $H_i$, $i\in [n_G]$, and the relevant fans. We consider $v_i\in V(G)$ with the relevant graph $H_i$, for consecutive $i\in[n_G]$. We apply one of the following cases, depending on the relation between degrees of $G$ and $H_i$.

\begin{description}
\item[Case 1.] $\Delta(G) > \Delta(H_i)$ or $\Delta(G) = \Delta(H_i)$, but $t_i<t$.

We color vertices and edges of $H_i$ in a proper way with $\Delta(H_i)+t_i$ colors from the set $[\Delta(H)+t_i+1]$, but we do not use the color assigned to $v_i$ in $f|_G$. Since $\Delta(H_i)+t_i <\Delta(G)+t$, it is doable. Finally, we use colors $\Delta(G)+t+1, \ldots, \Delta(G)+n_{H_i}+t$ to color $n_{H_i}$ edges in the fan $F_{v_i}$. Note, that these $n_{H_i}$ colors are not used in either $G$ or $H_i$.

\item[Case 2.] $\Delta(G) = \Delta(H_i)$ and $t_i=t$.

If we are able to color graph $H_i$ with $\Delta(H_i)+t$ colors in a proper total way such  that color $f(v_i)$ is not used to color vertices in $H_i$, then we do it. Next, we assign colors $\Delta(G)+t+1, \ldots, \Delta(G)+n_{H_i}+t$ to color $n_{H_i}$ edges in the fan $F_{v_i}$. Otherwise, we consider any total $(\Delta(H_i)+t)$-coloring of $H_i$. Note that taking exactly such a coloring of $H_i$ to $G\Tilde{\circ} \Lambda _{i=1}^n H_i$ results in improper partial total coloring. The vertices in $H_i$ that are assigned color $f(v_i)$ need to be recolored into a new color $\Delta(G)+t+1$ and such a modified total coloring of $H_i$ is taken to $G\Tilde{\circ} \Lambda _{i=1}^n H_i$.
Because $n_{H_i} \geq 2$ and $H_i$ is connected, then we are able to choose one vertex in $H_i$ not colored with $\Delta(G)+t+1$, let us say vertex $u'$, and the edge $v_iu'$ in the fan $F_{v_i}$ can be assigned color $\Delta(G)+t+1$. The rest of edges in the fan $F_{v_i}$ are colored with $\Delta(G)+t+2, \ldots, \Delta(G)+n_H+t$ colors, not used earlier either in $G$ or in $H_i$.

\end{description}

Finally, after coloring all outer graphs and the relevant fans, the total coloring of the generalized corona is proper. We claim that it is adjacent vertex distinguishing.
In order to justify this we consider the following cases.
\begin{itemize}
    \item Let $a$ and $b$ be two adjacent vertices in $H_i$, i.e. $ab\in E(H_i)$ for any $i\in[n_G]$. Note that $C^{H_i}(a)$, as well as $C^{H_i}(b)$, is completed only by one color in the whole avd total coloring of $G\Tilde{\circ} \Lambda _{i=1}^n H_i$. In Case 1 the colors used to color edges in $F_{v_i}$ were different from those in $C^{H_i}(a) \cup C^{H_i}(b)$ and they have not been used in $H$ before. In Case 2, at least one of these two vertices $a$ and $b$ is completed by a new color not used earlier. 
    Thus, $C^{G\Tilde{\circ} \Lambda _{i=1}^n H_i}(a)\neq C^{G\Tilde{\circ} \Lambda _{i=1}^n H_i}(b)$.
    \item Let $a$ and $b$ be any two adjacent vertices in $G$. We colored them taking into account only graph $G$, i.e. $C^G(a)\neq C^G(b)$, next their color sets were completed by the set of new colors, not used earlier in $G$. Thus, $C^{G\Tilde{\circ} \Lambda _{i=1}^n H_i}(a)\neq C^{G\Tilde{\circ} \Lambda _{i=1}^n H_i}(b)$.
    \item Let $a$ be a vertex of $H_i$ and let  $b$ be a vertex of $G$. Since in the corona $G\Tilde{\circ} \Lambda _{i=1}^n H_i$, $\deg(a)\leq n_{H_i} $ while $\deg(b)\geq n_{H_i}+1$, then $\deg(a) \neq \deg(b)$ and 
    $C^{G\Tilde{\circ} \Lambda _{i=1}^n H_i}(a)\neq C^{G\Tilde{\circ} \Lambda _{i=1}^n H_i}(b)$.
\end{itemize}

\end{proof}

As a consequence of the previous theorem, for the case when $H_i \simeq H_j$ for any $i,j \in [n_G]$, we get the following corollary.
\begin{corollary}
Let $G$ and $H$ be connected simple graphs on at least two vertices, for which $\chi_{a}''(G)\leq \Delta (G)+t$ with $t\geq 2$ and $\chi''(H)\leq \Delta (H)+t'$ with $1 \leq t'\leq t$. If $\Delta(G)\geq \Delta(H)$, then $$\chi_{a}''(G \circ H)\leq \Delta(G \circ H)+t.$$\label{theo:crazy}
\end{corollary} 


\begin{corollary}Let $G$ and $H$ be connected simple graphs on at least two vertices, for which $\Delta(G)\geq \Delta(H)$. 
\begin{enumerate}
    \item If Conjecture~\ref{conj:zhang} holds for $G$ and Conjecture~\ref{tcc:conj} holds for $H$, then $$\chi_{a}''(G \circ H)\leq \Delta(G \circ H)+3.$$
    \item If Conjecture~\ref{conj:zhang} holds for $G$ and $H$, then $$\chi_{a}''(G \circ H)\leq \Delta(G \circ H)+3.$$\label{theo:HlessG}
        \item If $G$ and $H$ are bipartite graphs. Then, $$\chi_a''(G \circ H) \leq \Delta(G \circ H)+2.$$\hfill $\Box$
\end{enumerate}
\end{corollary}

Observe that $\Delta(G\circ^l H)=\Delta(G)+l\cdot n_H$, for any $l\geq 1$. 
If $\Delta(H) \leq \Delta(G)$ then 
we immediately have $\Delta(H) \leq \Delta(G \circ^{l-1} H)$, for any $l\geq 2$. Hence we have the following generalization of Corollary \ref{theo:crazy}.

\begin{corollary} Let $G$ and $H$ be connected simple graphs on at least two vertices, for which $\Delta(G)\geq \Delta(H)$. 
\begin{enumerate}
\item If Conjecture~\ref{conj:zhang} holds for $G$ and Conjecture~\ref{tcc:conj} holds for $H$, then $$\chi_{a}''(G \circ^l H)\leq \Delta(G \circ^l H)+3,$$
for any integer $l\geq 2$.
    \item If Conjecture~\ref{conj:zhang} holds for $G$ and $H$, then $$\chi_{a}''(G \circ^l H)\leq \Delta(G \circ^l H)+3,$$ for any integer $l\geq 2$.
    \item If $G$ and $H$ are bipartite graphs, then $$\chi_{a}''(G \circ^l H)\leq \Delta(G \circ^l H)+2,$$ for any integer $l\geq 2$.\hfill $\Box$
\end{enumerate}
\end{corollary}




One can ask what about the Conjecture \ref{conj:zhang} for coronas $G\circ H$ where $\Delta(H)>\Delta(G)$. We also ask how big the difference between maximum degrees can be to be sure that AVDTC Conjecture holds. We partially answer these questions.


\begin{theorem}
Let $G$ and $H$ be connected simple graphs on at least two vertices, for which Conjecture~\ref{conj:zhang} holds. Let $\Delta(H)= \Delta(G)+1$. Then, $$\chi_a''(G\circ H) \leq \Delta(G \circ H)+3.$$\label{thm:diff1}
\end{theorem}
\begin{proof}
If $H$ is a bipartite graph, then, due to Theorem \ref{avd-bip}, we may start with any avd total $(\Delta(H)+2)$-coloring of each copy of $H$. In this case we need the same number of colors for avd total $(\Delta(G)+3)$-coloring of $G$. Hence, this case can be seen as equivalent to the one given in Case 2 in the proof of Theorem \ref{theo:crazy}. So, let us assume $H$ is not bipartite. Then an avd total $(\Delta(G \circ H)+3)$-coloring $f$ of $G\circ H$ can be obtained as follows.
\begin{enumerate}
    \item Color vertices and edges of graph $G$ with $\Delta(G)+3$ colors in adjacent vertex distinguishing way. We will refer to this part of the coloring as to $f|_{G}$. 
    \item 
    We will extend our avd total coloring $f|_{G}$ into avd total coloring of each copy of $H$ in $G\circ H$. 
    Let $v\in V(G)$, $c=f(v)$, and we consider the relevant copy of $H$.
    Let $f|_H$ denote an avd total $(\Delta (H)+3)$-coloring of $H$ such that there is a vertex $u\in V(H)$ for which $f(u)\neq c$ and the color $\Delta (H)+3$ does not belong to color set of $u$ in $H$, i.e. $\Delta (H)+3 \not \in C_{f|_H}^H(u)$.
    
    Note that such a coloring of $H$ always exists due to the fact that we have at least two missing colors in a color set of every vertex in $H$. So we color the chosen $H$, $H\subset G\circ H$, in the desirable way.
    
    Since $\Delta (H)+3=\Delta (G)+4$, the color $\Delta (H)+3$ does not belong to $C_{f|_H}^H(u)$, and $\Delta (G)+4$ is not used in $f|_G$, we can color an edge $uv$ in the fan $F_v$ with $\Delta (G)+4$.
     
    Note that after this step the partial total coloring $f$ of $G\circ H$ may not be proper. We need to fix it, if it is the case. We do it in the following way. If the coloring is improper, i.e. there are vertices in $H$ colored with $c$, we recolor them to $\Delta(G)+5$. Note that after this recoloring the new coloring, limited to graph $H$, is certainly proper avd-total-coloring, while the partial total coloring of the whole corona, received so far, is proper. 
    
    \item Next, we will complete the coloring of the fan $F_v$. If the coloring was initially not proper, we choose one vertex $w\in V(H)$ such that $w\neq u$ and $w$ is not colored with $c$. Such a vertex certainly exists, because $H$ is not bipartite. Note that $vw$ can be colored with $\Delta(G)+5$. We do it. The rest uncolored $n_H-2$ edges in $F_v$ are colored with new colors: $\Delta(G)+6, \ldots, \Delta(G)+n_H+3$. Otherwise, we color all uncolored $n_H-1$ edges in $F_v$ with $\Delta(G)+5, \ldots, \Delta(G)+n_H+3$.
     \end{enumerate}
    We repeat Step 2 and Step 3 for every $v\in V(G)$ with the relevant copy of $H$.

  Observe that for any vertex $v\in V(G)$, $C^{G\circ H}(v)=C^G(v)\cup \{ \Delta (G) +4,...,\Delta (G)+n_H+3\}$. Since the coloring $f|_{G}$ was adjacent vertex distinguishing, then also $f$, in accordance to vertices of $G$, is avd. In addition, for every $uw\in E(H)$, there exists at least one color $t\in \{1,2,..,\Delta (G)+3\}$ such that $t\notin C^H(u)\cap C^H(w)$. Hence, the obtained $(\Delta(G\circ H)+3)$-coloring of the whole corona $G\circ H$ is proper adjacent vertex distinguishing total coloring.
\end{proof}

\begin{theorem}
Let $G$ be a connected simple graph on at least two vertices, for which Conjecture~\ref{conj:zhang} holds. And let $H$ be the complete graph with order $n_H=\Delta (G)+ 3$.
Then $$\chi_a''(G\circ H) \leq \Delta(G \circ H)+3.$$ \label{thm:complete}
\end{theorem}
\begin{proof} 
We start from any adjacent vertex distinguishing total $(\Delta(G)+3)$-coloring of $G$. We will refer to this part of the coloring as to $f|_{G}$. An extension of $f|_{G}$ into the whole $(\Delta(G\circ H)+3)$-coloring $f$ of $G\circ H$ is obtained as follows. 

Let $v\in V(G)$, $f(v)=c$ with $c \in [\Delta(G)+3]$, which is equivalent to $c\in [\Delta(H)+1]$. We consider the relevant copy of $H$.

Let $f|_{H}$ be an adjacent vertex distinguishing total $(\Delta(H)+3)$-coloring of $H$, such that for all $u\in V(H)$, $f|_{H}(u)\notin \{ c, \Delta (G)+5\}.$ It is possible because $H=K_{\Delta (G)+3}$ and we use exactly $\Delta(G)+3$ colors to color vertices. 

Observe that for any vertex $u\in V(H)$, $|\overline{C^H(u)}|=2,$ and for any two vertices $u,w$ in $H$, the following holds $0\leq |\overline{C^H(u)} \cap \overline{C^H(w)}| \leq 1.$ 
We claim that there are two vertices  $u,w \in V(H)$ such that  $\overline{C^H(u)} \cap \overline{C^H(w)} =\emptyset$. Otherwise, consider two vertices $x_1,x_2\in V(H)$ such that $\overline{C^H(x_1)} =\{ a, d\},\overline{C^H(x_2)} =\{ b, d\}$. Now, the color $d$ need to belong to a color set of 
at least one vertex.
Let us say $x_3\in V(H)$ is such a vertex that $\overline{C^H(x_3)} =\{ a, b\}$, $a\neq d \neq b$, and for any other vertex $x_4\in V(H)$ we have only one of the following possibilities: $\overline{C^H(x_4)} =\{ a, b\}$, or $\overline{C^H(x_4)} =\{ a, d\}$, or $\overline{C^H(x_4)} =\{ d, b\}$, which is a contradiction. 

Let $\{ u,w\}\subset V(H)$ such that  $\overline{C^H(u)} \cap \overline{C^H(w)} =\emptyset $ with $\overline{C^H(u)} =\{ a, b \},$ $\overline{C^H(w)} =\{ d, g\}$, where $a,b,d,g$ are four different colors other than $c$. In this case, if it is necessary, we can recolor the graph $H$ such that $\overline{C^H(u)} =\{ a, \Delta (G) +4\},\overline{C^H(w)} =\{ d, \Delta (G)+5\}$. Next we color the edges $f(uv)=\Delta (G) + 4$,  $f(wv)=\Delta (G) + 5$ and we use the new colors  $\{ \Delta (G)+6, \Delta (G) + 7,..., \Delta (G)+n_H+3\}$ to the edges  $xv$, where $x\in V(H)\backslash \{ u, w\}.$

Observe that the obtained total $\Delta(G\circ H)+3)$-coloring of $\Delta(G\circ H)$ is adjacent vertex distinguishing. Indeed, due to the fact that $\deg(a)\neq \deg(b)$ for any pair of vertices such that $a\in V(G)$ and $b\in V(H)$, we need to consider only the case of different color sets for two adjacent vertices within graph $G$ or $H$. 
If $a$ and $b$ are two adjacent vertices in $G$, since $C^G(a)\neq C^G(b)$ and next their color sets were completed by the set of new colors not used earlier in $G$, then $C^{G\circ H}(a)\neq C^{G\circ H}(b)$.
On the other hand, if $a$ and $b$ are two adjacent vertices in $H$, then since their color sets from $f|_H$ were completed by one color and in the case where $u \neq a\neq w$ and $u \neq b\neq w$ it was a new color not used earlier in $H$, then we get $C^{G\circ H}(a)\neq C^{G\circ H}(b)$. The only doubt could appear for vertices $u,w$ in $H$, chosen as above. But since $\overline{C^H(u)} \cap \overline{C^H(w)} =\emptyset$, then $\overline{C^{G\circ H)}(u)}\cap [\Delta(G)+5]=a$ and $\overline{C^{G\circ H)}(w)}\cap [\Delta(G)+5]=d$, $a\neq d$. Thus $C^{G\circ H}(u)\neq C^{G\circ H}(w)$ and the proof is complete.

\end{proof}
\begin{proposition}
Let $H$ be a connected simple graph with $\Delta(H)\geq 3$ for which Conjecture \ref{conj:zhang} holds. Let  $\alpha(H)\geq 2$ and let $u_1$ and $u_2$ be any two non-adjacent vertices in $H$. Then for any color $c\in [\Delta(H)+1]$ there is an avd total $(\Delta(H)+3)$-coloring $f$ of $H$ such that all four conditions hold:
\begin{enumerate}
    \item $\Delta(H)+2 \in \overline{C^H_f(u_1)}$,
    \item $\Delta(H)+3 \in \overline{C^H_f(u_2)}$,
    \item $f(u_1) \neq c$,
    \item $f(u_2) \neq c$.
\end{enumerate}\label{prop:demand}
\end{proposition}
\begin{proof}
We can easily start from any avd total $(\Delta(H)+3)$-coloring $f$ of $H$ such that $f(u_1) =c_1 \neq c$ and $f(u_2)=c_2 \neq c$ and $c_1, c_2\in [\Delta(H)+1]$. It may happen that $c_1=c_2$. 

Now, let us assume that $\overline{C^H_f(u_1)}$ does not contain $\Delta(H)+2$. Since each vertex has at least two missing colors, let us say $\{a_1,a_2\}\subset \overline{C^H_f(u_1)}$, then we can exchange one of missing colors, let us say $a_1$, with $\Delta(H)+2$, and vice versa, in the whole graph $H$. 
Similarly, if $\overline{C^H_f(u_2)}$ does not contain $\Delta(H)+3$, let $\{b_1,b_2\}\subset \overline{C^H_f(u_2)}$. We can choose one of missing colors, different than $a_1$, and exchange it into $\Delta(H)+3$, and vice versa, also in the whole graph $H$.

Finally, we receive an avd total $(\Delta(H)+3)$-coloring $f$ of $H$ fulfilling the desirable conditions.
\end{proof}

\begin{theorem}
Let $G$ an $H$ be connected simple graphs on at least two vertices, for which Conjecture~\ref{conj:zhang} holds. Let $\Delta(H)= \Delta(G)+2$ and $\alpha(H)\geq 2$.
Then $$\chi_a''(G\circ H) \leq \Delta(G \circ H)+3.$$
\label{thm:diff2}
\end{theorem}
\begin{proof}
We start from any avd total $(\Delta(G)+3)$-coloring of $G$. We will refer to this part of the coloring as to $f|_{G}$. Further, we extend this coloring into the whole corona depending on the form of $H$. 

Let $v\in V(G)$ and $f|_G(v)=c$, where $c\in[\Delta(G)+3]$, or in other words $c\in [\Delta(H)+1]$. We consider the relevant copy of $H$. 
If $H$ is bipartite, $H=(V_1\cup V_2,E)$, then by Theorem \ref{avd-bip}, $\chi_a''(H) \leq  \Delta(H)+2$. By K\"{o}nig's theorem, we color the edges of $H$ with $\Delta(H)$ colors: color $c$ and $\Delta(H)-1$ other colors from $[\Delta(H)+1]$.
One color among $\{1,\ldots,\Delta(H)+1\}$ is not used to color edges in $H$. We assign it to color all vertices in $V_1$, while $\Delta(G)+4$ is used to color all vertices in $V_2$.
Next, we choose one vertex $x\in V_1$ and assign $\Delta(G)+4$ to $vx$. Further, we complete the coloring of edges in the fan $F_v$ by coloring the residual uncolored $n_H-1$ edges with different colors from the set 
$\{\Delta(G)+5,\ldots, \Delta(G)+n_H+3\}$. 
We repeat the procedure for all copies of $H$. It is easy to see that the coloring $f$ of $G\circ H$ is adjacent vertex distinguishing total coloring. 

Otherwise, i.e. if $H$ is not bipartite, an extension of $f|_{G}$ into the whole $(\Delta(G\circ H)+3)$-coloring $f$ of $G\circ H$ is obtained as follows. Let $I$ be any independent vertex set in $H$ of size at least 2 and let $u_1,u_2$ be any two vertices from $I$. 

\begin{enumerate}
    \item Color vertices and edges of $H$ in avd total way with $\Delta(H)+3=\Delta(G)+5$ colors in such a way that color $\Delta(G)+4$ is missing in the color set of $u_1$, i.e. $\Delta(G)+4\in \overline{C^H(u_1)}$, and $\Delta(G)+5 \in \overline{C^H(u_2)}$, and $f(u_1)\neq c$, and $f(u_2)\neq c$. This is possible due to Proposition \ref{prop:demand}.
    
    It may happen that the partial total coloring of $G\circ H$ is improper at this stage. We will fix it later.
    \item Assign color $\Delta(G)+4$ to an edge $vu_1$ and color $\Delta(G)+5$ to $vu_2$. 
    \item If the partial total coloring of $G\circ H$ is improper then we recolor vertices colored initially with $c$ into $\Delta(G)+6$.
    
    \item Choose any vertex $x\in V\backslash\{u_1,u_2\}$ not colored with $\Delta(G)+6$ and assign $\Delta(G)+6$ to $vx$. 
    
    Note that such a vertex always exists because $H$ is not bipartite and there exists at least one edge with both endvertices in $V\backslash H$. At least one of them is not colored with $\Delta(G)+6$.
    
    \item Complete the coloring of edges in the fan $F_v$ by coloring the remaining uncolored $n_H-3$ edges with new colors $\Delta(G)+7,\ldots, \Delta(G)+n_H+3$. 
    
    Note, that the only cases for which we really need to check the color sets are adjacent vertices of $x$ which are colored with $\Delta(G)+6$. But edges in the fan $F_v$ joining such a vertex are colored with completely new colors not used any more in $F_v$ and in the copy of $H$ under consideration. So, even for such adjacent vertices, their color sets are different. Thus, the partial coloring of $G\circ H$ is proper and adjacent vertex distinguishing from the point of view of vertices from the copy of $H$. 
\end{enumerate}

We repeat the same procedure for all copies of $H$. Finally, we get an avd total $\Delta(G\circ H)+3$-coloring of $G\circ H$. The justification for two adjacent vertices within $G$ is the same as in the proofs of the previous theorems. The proof is complete. 
\end{proof}

Up to now, since actually Theorem \ref{thm:complete} concerns the case when $\alpha(H)=1$ and $\Delta(H)=\Delta(G)+2$, we proved Conjecture \ref{conj:zhang} for all coronas $G \circ H$ of graphs $G$ and $H$ with $\Delta(H)-\Delta(G)\leq 2$, under some additional assumptions for $G$ and $H$. Now we present a partial result concerning graphs with a greater difference, i.e. let $\Delta(H)=\Delta(G)+k$, $k\geq 3$. We start with $H$ being  bipartite and $k=3$.

\begin{theorem}
Let $G$ be connected simple graphs on at least two vertices, for which Conjecture~\ref{conj:zhang} holds. Let $H=(V_1\cup V_2,E)$ be a bipartite graph with $\Delta(H)= \Delta(G)+3$. Then $$\chi_a''(G\circ H) \leq \Delta(G \circ H)+3.$$\label{thm:bip}
\end{theorem}
\begin{proof}
By Theorem \ref{avd-bip}, $\chi_a''(H) \leq  \Delta(H)+2$. We start from any $(\Delta(G)+3)$-adjacent-vertex-distinguishing total coloring of $G$. 
Let $v\in V(G)$ and $f|_G(v)=c$, where $c\in[\Delta(G)+3]$, or in other words $c\in [\Delta(H)+6]$. We consider the relevant copy of $H$. By K\"{o}nig's theorem, we color the edges in the relevant copy of $H$ with $\Delta(H)$ colors, $\Delta(H)= \Delta(G)+3$, among them the color $c$. We assign color $\Delta(G)+4$ to all vertices in $V_1$ while color $\Delta(G)+5$ is assigned to all vertices from $V_2$. Now, let us choose a vertex colored with $\Delta(G)+4$, let us name it $x_1$, $x_1\in V_1$, and a vertex colored with $\Delta(G)+5$, let us name it by $x_2$, $x_2\in V_2$. 
At first, try to choose non-adjacent vertices or vertices of different degrees. When it is not possible, choose any two vertices colored appropriately. Let $f(vx_1)=\Delta(G)+5$ and let $f(vx_2)=\Delta(G)+4$. If $C^{G\circ H}(x_1)=C^{G\circ H}(x_2)$, then recolor $x_2$ into $\Delta(G)+6$ and choose any vertex in the same partition, let us say $x_3$, $x_3\in V_2$, and color $vx_3$ with $\Delta(G)+6$, otherwise we assign $\Delta(G)+6$ to any uncolored edge in the $F_v$. 
Next, we complete the coloring of edges in the fan $F_v$ by coloring the remaining uncolored $n_H-3$ edges with new colors $\Delta(G)+7,\ldots, \Delta(G)+n_H+3$. 
Note that the partial coloring of $G\circ H$ is proper and adjacent vertex distinguishing from the point of view of vertices from the copy of $H$.

We repeat the same procedure for all vertices $v\in V(G)$ and the relevant copies of $H$. 

Since we completed $C^G_f(v)$, for each $v\in V(G)$, by the same set of colors $\{\Delta(G)+4, \ldots,\Delta(G)+n_H+3\}$, and taking into account the previous reasoning for color sets of any two adjacent vertices from $H$, the total coloring of the whole corona is adjacent vertex distinguishing. The proof is complete.
\end{proof}

Note that a similar idea applied to $H$ being bipartite graph, in particular complete bipartite graph, with $\Delta(H)=\Delta(G)+k$ for $k\geq 4$ will not work. We mean assigning only $\Delta(H)$ colors to edges of $H$ and assigning colors $\Delta(G)+4, \ldots, \Delta(G)+n_H+3$ to edges of a fan $F_v$. Since in such solutions all $\Delta(H)$ colors used to color edges of $H$ are present in color sets of all vertices in $H$, then we can use none of these colors to color edges in $F_v$ and thus such an approach would involve more than $\Delta(G\circ H)+3$ colors, for $k\geq 4$. Hence, we certainly need to color edges of $H$ with more colors than only $\Delta(H)$.

Now, let us consider more general graphs. We take an attempt of generalization of the method given in the proof of Theorem \ref{thm:diff2}. The basis of this method is such an avd total $(\Delta(H)+3)$-coloring $f|_H$ of $H$ that there is a set of non-adjacent vertices $U=\{u_1,\ldots,u_k\}$ for which the colors from $[\Delta(H)+3]\backslash [\Delta(G)+3]$ are missing colors for vertices in $U$. For $k=2$ such a coloring was easy to achieve  and it was guaranteed by Proposition \ref{prop:demand}. For a greater $k$ we need additional conditions for degrees of $u\in U$.

\begin{proposition}
Let $H$ be a connected simple graph with $\Delta(H)\geq k+1$ for which Conjecture \ref{conj:zhang} holds. Let $\alpha(H) \geq k$ and let $u_1,\ldots,u_k$ be any $k$ non-adjacent vertices in $H$ such that $\deg(u_1)\leq \Delta(H)$, $\deg(u_2)\leq \Delta(H)$, $\deg(u_i)\leq \Delta(H)-i+2$ for $i\in\{3,\ldots,k-2\}$.
Then for any color $c\in [\Delta(H)-k+4]$ there is an avd total $(\Delta(H)+3)$-coloring $f$ of $H$ such that all the following conditions hold:
\begin{enumerate}
    \item $\Delta(H)+1+i \in \overline{C^H_f(u_i)}$, $i\in[k]$,
    \item $f(u_i)\neq c$, $i\in[k]$.
\end{enumerate}\label{prop:demandk}
\end{proposition}
\noindent We remain the full proof for the reader, but it is easy to see that the conditions for degrees of chosen vertices in an independent set of size at least $k$ guarantee us that $|\bigcup_{i\in[k]}\overline{C^H_f(u_i)}|\geq k$. So we are able to exchange colors in $H$ to achieve the desirable conditions. Of course, the conditions for degrees in Proposition \ref{prop:demandk} are not the only one guaranteeing us ''good'' avd total coloring of $H$.

\begin{theorem}
Let $G$ an $H$ be connected simple graphs on at least two vertices, for which Conjecture~\ref{conj:zhang} holds. Let  $k\geq 3$ be an integer, $\Delta(H)= \Delta(G)+k$, and $\alpha(H)\geq k$. If there exist  $k$ non-adjacent vertices in $H$ $u_1,\ldots,u_k$ such that $\deg(u_1)\leq \Delta(H)$, $\deg(u_2)\leq \Delta(H)$, $\deg(u_i)\leq \Delta(H)-i+2$ for $i\in\{3,\ldots,k-2\}$ then
 $$\chi_a''(G\circ H) \leq \Delta(G \circ H)+3.$$
\label{thm:diffk}
\end{theorem}
\begin{proof}
We start from any adjacent vertex distinguishing total $(\Delta(G)+3)$-coloring of $G$. We will refer to this part of the coloring as to $f|_{G}$. An extension of $f|_{G}$ into the whole $(\Delta(G\circ H)+3)$-coloring $f$ of $G\circ H$ is obtained as follows. Let $u_1,\ldots,u_k$ be any $k$ non-adjacent vertices fulfilling the assumption of the theorem. 

\begin{enumerate}
    \item Color vertices and edges of $H$ in avd total way with $\Delta(H)+3=\Delta(G)+k+3$ colors in such a way that color $\Delta(G)+3+i$ is missing in the color set of $u_i$, i.e. $\Delta(G)+3+i\in \overline{C^H(u_i)}$, $i\in[k]$, and none of vertices $u_i,\ldots,u_k$ is colored with $c$. This is possible due to Proposition \ref{prop:demandk}.
    
    It may happen that the partial total coloring of $G\circ H$ is improper at this stage. We will fix it later.
    \item Assign color $\Delta(G)+3+i$ to an edge $vu_i$, for every $i\in [k]$.
    \item If the partial total coloring of $G\circ H$ is improper then we recolor vertices colored initially with $c$ into $\Delta(G)+k+4$.
    
    \item Choose any vertex $x\in V\backslash\{u_1,\ldots,u_k\}$ not colored with $\Delta(G)+k+4$ and assign $\Delta(G)+k+4$ to $vx$. Note that such a vertex always exists. Since $\Delta(H)\geq k+1$ there must exist an edge whose endvertices are out of the set $\{u_1,\ldots,u_k\}$, then at least one of the endvertices of such an edge is not colored with $\Delta(G)+k+4$.

    \item Complete the coloring of edges in the fan $F_v$ by coloring the remaining uncolored $n_H-k-1$ edges with new colors $\Delta(G)+k+5,\ldots, \Delta(G)+n_H+3$. Note, that the only cases for which we really need to check the color sets are adjacent vertices of $x$ which are colored with $\Delta(G)+k+4$. But edges in the fan $F_v$ joining such vertices are colored with completely new colors not used any more in $F_v$ and in the copy of $H$ under consideration. So, even for such adjacent vertices, their color sets are different. Thus, the partial coloring of $G\circ H$ is proper and adjacent vertex distinguishing from the point of view of vertices from the copy of $H$.
\end{enumerate}
We repeat the same procedure for all copies of $H$. Finally, we get a total $(\Delta(G\circ H)+3)$-coloring of $G\circ H$.
Since we completed $C^G_f(v)$, for each $v\in V(G)$, by the same set of colors $\{\Delta(G)+4, \ldots,\Delta(G)+n_H+3\}$, the total coloring of the whole corona is adjacent vertex distinguishing. The proof is complete.
\end{proof}

\section{Conclusion}
In the paper we considered adjacent vertex distinguishing total coloring of corona graphs in the context of AVDTC Conjecture posted by Zhang in 2005. We confirmed this conjecture for:
\begin{itemize}
    \item generalized coronas $G\Tilde{\circ} \Lambda _{i=1}^n H_i$ with $\Delta(G)\geq \Delta(H_i)$, under the assumption that Conjecture \ref{conj:zhang} holds for $G$ and Conjecture \ref{tcc:conj} holds for every $H_i$, $i\in[n_G]$;
    \item all simple coronas $G \circ H$ with $\Delta(H)-\Delta(G)\leq 2$, under the assumption that Conjecture \ref{conj:zhang} holds for $G$ and $H$;
    
    Actually, the assumption for graph $H$ can be a little weaker. Our proofs show that it is enough that Conjecture \ref{tcc:conj} holds for $H$.
    
    \item all simple coronas $G\circ H$ with $\Delta(H)=\Delta(G)+3$, where $H$ is bipartite and Conjecture \ref{conj:zhang} holds for $G$;
    \item some simple coronas $G\circ H$ with $\Delta(H)= \Delta(G)+k$, $k\geq 3$ under some additional constraints - for details see Theorem \ref{thm:diffk}.
\end{itemize}

Taking into account the results known from the literature and taking the results from this work, we can replace our general graphs $G$ and $H$ with particular graph classes fulfilling Conjecture \ref{tcc:conj} and \ref{conj:zhang}. In Table \ref{HmnG} we present only exemplary results.

\begin{table}[htb]
\begin{center}
\begin{tabular}{|c|*{4}{c|}}\hline

\backslashbox[50mm]{$G$}{$H$} & path & cycle & 3-regular& 4-regular \\ \hline

path & \checkmark & \checkmark & \checkmark&\checkmark\\ \hline
cycle & \checkmark & \checkmark & \checkmark&\checkmark\\ \hline
3-regular & \checkmark& \checkmark& \checkmark &  \checkmark\\\hline
4-regular & \checkmark & \checkmark& \checkmark & \checkmark \\\hline
complete graph $K_n$, $n\geq 6$ & \checkmark & \checkmark & \checkmark& \checkmark\\ \hline
\end{tabular}
\caption{An exemplary graph classes of $G$ and $H$ such that Conjecture \ref{conj:zhang} holds for $G\circ H$.}
\end{center}
\label{HmnG}
\end{table}

One can ask what about the remaining coronas not covered by the results of this paper. We retain this for a further investigation and as an open problem for other graph theorists.

\newpage
\noindent{\bf Compliance with Ethical Standards:}

\noindent{\bf  Funding:} This study was funded by UNAM (Grant  PAPIIT-UNAM-IN117219) - Author B.

\noindent{\bf Ethical approval:} This article does not contain any studies with human participants or animals performed by any of the authors.

\noindent{\bf Statements and Declarations:} The manuscript has no associated data.

\bibliography{bibliography_avdtc}

\begin{thebibliography}{23}
\expandafter\ifx\csname natexlab\endcsname\relax\def\natexlab#1{#1}\fi
\providecommand{\url}[1]{\texttt{#1}}
\providecommand{\href}[2]{#2}
\providecommand{\path}[1]{#1}
\providecommand{\DOIprefix}{doi:}
\providecommand{\ArXivprefix}{arXiv:}
\providecommand{\URLprefix}{URL: }
\providecommand{\Pubmedprefix}{pmid:}
\providecommand{\doi}[1]{\href{http://dx.doi.org/#1}{\path{#1}}}
\providecommand{\Pubmed}[1]{\href{pmid:#1}{\path{#1}}}
\providecommand{\bibinfo}[2]{#2}
\ifx\xfnm\relax \def\xfnm[#1]{\unskip,\space#1}\fi
\bibitem[{Behzad et~al.(1967)Behzad, Chartrand and Jr}]{behzad}
\bibinfo{author}{Behzad, M.}, \bibinfo{author}{Chartrand, G.},
  \bibinfo{author}{Jr, J.C.}, \bibinfo{year}{1967}.
\newblock \bibinfo{title}{The colour numbers of complete graphs}.
\newblock \bibinfo{journal}{J. Lond. Math. Soc.} \bibinfo{volume}{42},
  \bibinfo{pages}{226--228}.
\bibitem[{Chang et~al.(2020)Chang, Hu, Wang and Yu}]{planar8}
\bibinfo{author}{Chang, Y.}, \bibinfo{author}{Hu, J.}, \bibinfo{author}{Wang,
  G.}, \bibinfo{author}{Yu, X.}, \bibinfo{year}{2020}.
\newblock \bibinfo{title}{Adjacent vertex distinguishing total coloring of
  planar graphs with maximum degree $8$}.
\newblock \bibinfo{journal}{Discrete Math.} \bibinfo{volume}{343},
  \bibinfo{pages}{112014}.
\bibitem[{Chen(2008)}]{delta3}
\bibinfo{author}{Chen, X.}, \bibinfo{year}{2008}.
\newblock \bibinfo{title}{On the adjacent vertex distinguishing total coloring
  numbers of graphs with $\delta=3$}.
\newblock \bibinfo{journal}{Discrete Math.} \bibinfo{volume}{308}.
\bibitem[{Chen and Zhang(2006)}]{cart}
\bibinfo{author}{Chen, X.}, \bibinfo{author}{Zhang, Z.}, \bibinfo{year}{2006}.
\newblock \bibinfo{title}{Adjacent-vertex-distinguishing total chromatic number
  of $p_m \times k_n$}.
\newblock \bibinfo{journal}{J. Math. Research \& Exposition}
  \bibinfo{volume}{26}, \bibinfo{pages}{489--494}.
\bibitem[{Chen et~al.(2008)Chen, Zhang and Sang}]{cartpath}
\bibinfo{author}{Chen, X.}, \bibinfo{author}{Zhang, Z.}, \bibinfo{author}{Sang,
  Y.}, \bibinfo{year}{2008}.
\newblock \bibinfo{title}{A note on adjacent-vertex-distinguishing total
  chromatic numbers for $p_m \times p_n$, $p_m \times c_n$ and $c_m \times
  c_n$}.
\newblock \bibinfo{journal}{J. Math. Research \& Exposition}
  \bibinfo{volume}{28}, \bibinfo{pages}{789--798}.
\bibitem[{Cheng et~al.(2016)Cheng, Wang and Wu}]{planar10}
\bibinfo{author}{Cheng, X.}, \bibinfo{author}{Wang, G.}, \bibinfo{author}{Wu,
  J.}, \bibinfo{year}{2016}.
\newblock \bibinfo{title}{The adjacent vertex distinguishing total chromatic
  numbers of planar graphs with $\delta = 10$}.
\newblock \bibinfo{journal}{J. Comb. Optim.} , \bibinfo{pages}{1–15}.
\bibitem[{Coker and Johannson(2012)}]{coker}
\bibinfo{author}{Coker, T.}, \bibinfo{author}{Johannson, K.},
  \bibinfo{year}{2012}.
\newblock \bibinfo{title}{The adjacent vertex distinguishing total chromatic
  number}.
\newblock \bibinfo{journal}{Discrete Math.} \bibinfo{volume}{312},
  \bibinfo{pages}{2741--2750}.
\bibitem[{Frucht and Harary(1970)}]{frucht}
\bibinfo{author}{Frucht, R.}, \bibinfo{author}{Harary, F.},
  \bibinfo{year}{1970}.
\newblock \bibinfo{title}{On the corona of two graphs}.
\newblock \bibinfo{journal}{Aequationes Math.} \bibinfo{volume}{4},
  \bibinfo{pages}{322--325}.
\bibitem[{Furma\'nczyk and Kubale(2016)}]{harder}
\bibinfo{author}{Furma\'nczyk, H.}, \bibinfo{author}{Kubale, M.},
  \bibinfo{year}{2016}.
\newblock \bibinfo{title}{Equitable coloring of corona products of cubic graphs
  is harder than ordinary coloring}.
\newblock \bibinfo{journal}{Ars Mathematica Contemporanea}
  \bibinfo{volume}{10}, \bibinfo{pages}{333--347}.
\bibitem[{Hu et~al.(2019)Hu, Wang, Wu, Yang and Yu}]{planar9}
\bibinfo{author}{Hu, J.}, \bibinfo{author}{Wang, G.}, \bibinfo{author}{Wu, J.},
  \bibinfo{author}{Yang, D.}, \bibinfo{author}{Yu, X.}, \bibinfo{year}{2019}.
\newblock \bibinfo{title}{Adjacent vertex distinguishing total coloring of
  planar graphs with maximum degree 9}.
\newblock \bibinfo{journal}{Discrete Math.} \bibinfo{volume}{342},
  \bibinfo{pages}{1392--1402}.
\bibitem[{Huang and Wang(2012)}]{planfrom10}
\bibinfo{author}{Huang, D.}, \bibinfo{author}{Wang, W.}, \bibinfo{year}{2012}.
\newblock \bibinfo{title}{Adjacent vertex distinguishing total coloring of
  planar graphs with large maximum degree}.
\newblock \bibinfo{journal}{Sci. Sin. Math.} \bibinfo{volume}{42},
  \bibinfo{pages}{151--164}.
\bibitem[{Huang et~al.(2012)Huang, Wang and Yan}]{huang}
\bibinfo{author}{Huang, D.}, \bibinfo{author}{Wang, W.}, \bibinfo{author}{Yan,
  C.}, \bibinfo{year}{2012}.
\newblock \bibinfo{title}{A note on the adjacent vertex distinguishing total
  chromatic number of graphs}.
\newblock \bibinfo{journal}{Discrete Math.} \bibinfo{volume}{312},
  \bibinfo{pages}{3544--3546}.
\bibitem[{Hulgan(2009)}]{short}
\bibinfo{author}{Hulgan, J.}, \bibinfo{year}{2009}.
\newblock \bibinfo{title}{Concise proofs for adjacent vertex-distinguishing
  total colorings}.
\newblock \bibinfo{journal}{Discrete Math.} \bibinfo{volume}{309},
  \bibinfo{pages}{2548--2550}.
\bibitem[{Li et~al.(2006)Li, Zhang, Wang, Wei and Yan}]{join_path_fan}
\bibinfo{author}{Li, J.}, \bibinfo{author}{Zhang, Z.}, \bibinfo{author}{Wang,
  Z.}, \bibinfo{author}{Wei, B.}, \bibinfo{author}{Yan, L.},
  \bibinfo{year}{2006}.
\newblock \bibinfo{title}{On the adjacent vertex-distinguishing equitable-total
  chromatic number of $p_m \lor f_n$}.
\newblock \bibinfo{journal}{International Journal of Pure and Applied
  Mathematics} \bibinfo{volume}{29}, \bibinfo{pages}{159--169}.
\bibitem[{Papaioannou and Raftopoulou(2014)}]{4reg}
\bibinfo{author}{Papaioannou, A.}, \bibinfo{author}{Raftopoulou, C.},
  \bibinfo{year}{2014}.
\newblock \bibinfo{title}{On the avdtc of 4-regular graphs}.
\newblock \bibinfo{journal}{Discrete Math.} \bibinfo{volume}{330},
  \bibinfo{pages}{20--40}.
\bibitem[{Rosenfeld(1970)}]{rosen}
\bibinfo{author}{Rosenfeld, M.}, \bibinfo{year}{1970}.
\newblock \bibinfo{title}{On the total coloring of certain graphs}.
\newblock \bibinfo{journal}{Israel J. Math.} \bibinfo{volume}{9},
  \bibinfo{pages}{396--402}.
\bibitem[{Sun and Sun(2007)}]{cart_myc}
\bibinfo{author}{Sun, Y.}, \bibinfo{author}{Sun, L.}, \bibinfo{year}{2007}.
\newblock \bibinfo{title}{The (adjacent) vertex-distinguishing total coloring
  of the mycielski graphs and the cartesian product graphs}, in:
  \bibinfo{booktitle}{Discrete Geometry, Combinatorics and Graph Theory},
  \bibinfo{publisher}{Springer}. pp. \bibinfo{pages}{200--205}.
\bibitem[{Vizing(1964)}]{Vizing}
\bibinfo{author}{Vizing, V.G.}, \bibinfo{year}{1964}.
\newblock \bibinfo{title}{On an estimate of the chromatic class of a p-graph}.
\newblock \bibinfo{journal}{Metody Diskret. Analiz.} , \bibinfo{pages}{25--30}.
\bibitem[{Wang(2017)}]{relations}
\bibinfo{author}{Wang, G.}, \bibinfo{year}{2017}.
\newblock \bibinfo{title}{Relation between cartesian product and adjacent
  vertex distinguishing coloring}.
\newblock \bibinfo{journal}{J. Zhejiang University} \bibinfo{volume}{44},
  \bibinfo{pages}{520--525}.
\bibitem[{Wang(2007)}]{delta3wang}
\bibinfo{author}{Wang, H.}, \bibinfo{year}{2007}.
\newblock \bibinfo{title}{On the adjacent vertex-distinguishing total chromatic
  numbers of the graphs with $\delta(g) = 3$}.
\newblock \bibinfo{journal}{J. Comb. Optimization} \bibinfo{volume}{14},
  \bibinfo{pages}{87--109}.
\bibitem[{Wang and Wang(2010)}]{outer}
\bibinfo{author}{Wang, Y.}, \bibinfo{author}{Wang, W.}, \bibinfo{year}{2010}.
\newblock \bibinfo{title}{Adjacent vertex distinguishing total colorings of
  outerplanar graphs}.
\newblock \bibinfo{journal}{J. Comb. Optimization} \bibinfo{volume}{19},
  \bibinfo{pages}{123--133}.
\bibitem[{Wang et~al.(2007)Wang, Yan and Zhang}]{join_paths_cycles}
\bibinfo{author}{Wang, Z.}, \bibinfo{author}{Yan, L.}, \bibinfo{author}{Zhang,
  Z.}, \bibinfo{year}{2007}.
\newblock \bibinfo{title}{Vertex distinguishing equitable total chromatic
  number of join graph}.
\newblock \bibinfo{journal}{Acta Mathematicae Applicatae Sinica, English
  Series} \bibinfo{volume}{23}, \bibinfo{pages}{433--438}.
\bibitem[{Zhang et~al.(2005)Zhang, Chen, Li, Yao, Lu and Wang}]{zhang}
\bibinfo{author}{Zhang, Z.}, \bibinfo{author}{Chen, X.}, \bibinfo{author}{Li,
  J.}, \bibinfo{author}{Yao, B.}, \bibinfo{author}{Lu, X.},
  \bibinfo{author}{Wang, J.}, \bibinfo{year}{2005}.
\newblock \bibinfo{title}{On the adjacent vertex distinguishing total coloring
  of graphs}.
\newblock \bibinfo{journal}{Sci. China Ser. A} \bibinfo{volume}{48},
  \bibinfo{pages}{289--299}.

\end{thebibliography}
\end{document}